\documentclass[12pt]{article}

\usepackage[T1]{fontenc}
\usepackage[utf8]{inputenc}
\usepackage{authblk}

\usepackage{amsmath, amssymb, amsfonts, amsbsy, amsthm, latexsym, color}
\usepackage{graphicx}
\usepackage{enumerate}
\everymath{\displaystyle}
\newtheorem{theorem}{Theorem}[section]
\newtheorem{proposition}[theorem]{Proposition}
\newtheorem{lemma}[theorem]{Lemma}

\newtheorem{definition}[theorem]{Definition}
\newtheorem{example}[theorem]{Example}

\newtheorem{remark}[theorem]{Remark}

\numberwithin{equation}{section}

\newcommand{\C}{{\mathbb{C}}}
\newcommand{\N}{{\mathbb{N}}}

\newcommand{\R}{{\mathbb{R}}}

\newcommand{\Z}{{\mathbb{Z}}}

\begin{document}

\title{Phase retrieval in infinite-dimensional Hilbert spaces}

\author[1]{Jameson Cahill} 
\author[2]{Peter G. Casazza \thanks{The second author was supported by  NSF DMS 1609760; NSF ATD 1321779; and ARO  W911NF-16-1-0008}}
\author[3]{Ingrid Daubechies \thanks{The third author was supported by AFOSR grant 00002113-02; ONR grant N00014-11-1-0714-06-7; and NSF grant DMS-1516988}}

\affil[1]{Department of Mathematical Sciences, New Mexico State University}
\affil[2]{Department of Mathematics, University of Missouri}
\affil[3]{Department of Mathematics, Duke University}

\date{\today}

\maketitle
\begin{abstract}
The main result of this paper states that phase retrieval in infinite-dimensional Hilbert spaces is never uniformly stable, in sharp contrast to the finite dimensional setting in which phase retrieval is always stable. This leads us to derive stability results for signals depending on how well they are approximated by finite expansions.
\end{abstract}

\vspace{0.35in}

\section{Introduction}

Given a separable Hilbert space $\mathcal{H}$, phase retrieval deals with the problem of recovering an unknown $f\in\mathcal{H}$ from a set of intensity measurements $(|\langle f,\varphi_n\rangle|)_{n\in I}$ for some countable collection $\Phi=\{\varphi_n\}_{n\in I}\subseteq\mathcal{H}$. Note that if $f=\alpha g$ with $|\alpha|=1$ then $|\langle f,\varphi_n\rangle|=|\langle g,\varphi_n\rangle|$ for every $n\in I$ regardless of our choice of $\Phi$; we say $\Phi$ \textit{does phase retrieval} if the converse of this statement is true, \textit{i.e.}, if the equalities $|\langle f,\varphi_n\rangle|=|\langle g,\varphi_n\rangle|$ for every $n$ imply that there is a unimodular scalar $\alpha$ so that $f=\alpha g$.

We will generally assume that $\Phi$ forms a \textit{frame} for $\mathcal{H}$, \textit{i.e.}, there are positive constants $0<A\leq B<\infty$ so that
$$
A\|f\|^2\leq\sum_{n\in I}|\langle f,\varphi_n\rangle|^2\leq B\|f\|^2
$$
for every $f$ in $\mathcal{H}$.  We call the operator $T_{\Phi}:\mathcal{H}\rightarrow\ell^2(I)$ given by
$$
T_{\Phi}(f)=(\langle f,\varphi_n\rangle)_{n\in I}
$$
the \textit{analysis operator} of $\Phi$. We denote by $\mathcal{A}_{\Phi}:\mathcal{H}\rightarrow\ell^2(I)$ the nonlinear mapping given by
$$
\mathcal{A}_{\Phi}(f)=(|\langle f,\varphi_n\rangle|)_{n\in I},
$$
so that $\Phi$ does phase retrieval if and only if $\mathcal{A}_{\Phi}$ is injective on $\mathcal{H}/\sim$ where $f\sim g$ if $f=\alpha g$ with $|\alpha|=1$.

\begin{definition}
We say a frame $\{\varphi_n\}_{n\in I}$ for a Hilbert space $\mathcal{H}$ has the complement property if for every subset $S\subseteq I$ we have
$$
\overline{\mathrm{span}}\{\varphi_n\}_{n\in S}=\mathcal{H}\text{ or }\overline{\mathrm{span}}\{\varphi_n\}_{n\not\in S}=\mathcal{H}.
$$
\end{definition}

\begin{theorem}\label{CP}
(a) Let $\mathcal{H}$ be a separable Hilbert space and let $\Phi$ be a frame for $\mathcal{H}$.  If $\Phi$ does phase retrieval then $\Phi$ has the complement property. \\ 
(b) Let $\mathcal{H}$ be a separable Hilbert space over the real numbers and let $\Phi$ be a frame for $\mathcal{H}$. If $\Phi$ has the complement property then $\Phi$ does phase retrieval.
\end{theorem}

\begin{proof}
(a) Suppose $\Phi$ does not have the complement property and find $S\subseteq I$ so that neither $\{\varphi_n\}_{n\in S}$ nor $\{\varphi_n\}_{n\not\in S}$ spans $\mathcal{H}$. Then we can find nonzero $f,g\in\mathcal{H}$ so that $\langle f,\varphi_n\rangle=0$ for all $n\in S$ and $\langle g,\varphi_n\rangle=0$ for all $n\not\in S$. Also, since $\Phi$ is a frame we know that $f\neq\lambda g$ for any scalar $\lambda$, so in particular we know $f+g\neq 0$ and $f-g\neq 0$. It now follows that $|\langle f+g,\varphi_n\rangle|=|\langle f-g,\varphi_n\rangle|$ for all $n\in I$ but $f+g\neq\lambda(f-g)$ for any scalar $\lambda$, so $\Phi$ does not do phase retrieval.

(b) Suppose $\Phi$ does not do phase retrieval and find nonzero $f,g\in\mathcal{H}$ so that $|\langle f,\varphi_n\rangle|=|\langle g,\varphi_n\rangle|$ for every $n\in I$, but $f\neq\pm g$. Since $\mathcal{H}$ is a real Hilbert space this means that $\langle f,\varphi_n\rangle=\pm\langle g,\varphi_n\rangle$, so let $S=\{n\in I:\langle f,\varphi_n\rangle=\langle g,\varphi_n\rangle\}$. Then $f-g\neq 0$ but $\langle f-g,\varphi_n\rangle=0$ for every $n\in S$ so $\overline{\mathrm{span}}\{\varphi_n\}_{n\in S}\neq\mathcal{H}$, and similarly $f+g\neq 0$ but $\langle f+g,\varphi_n\rangle=0$ for every $n\not\in S$ so $\overline{\mathrm{span}}\{\varphi_n\}_{n\not\in S}\neq\mathcal{H}$, which means $\Phi$ does not have the complement property.
\end{proof}

This theorem was originally proved in \cite{BCE06} where it was only stated in the finite-dimensional case, but the proof still holds in infinite dimensions without any modifications. The original proof of part (a) presented in \cite{BCE06} did not give the correct conclusion in the case where $\mathcal{H}$ is a Hilbert space over the complex numbers. This was observed by the authors of \cite{BCMN14} where they presented a much more complicated proof for this case. It turns out that the proof presented in \cite{BCE06} does hold in this case with only minor modifications, which is the proof presented above.

We remark here that recently several papers have been devoted to showing that certain frames do phase retrieval for infinite-dimensional spaces over both the real and complex numbers, so by Theorem \ref{CP} all of these frames have the complement property. For instance, in \cite{T11} it is shown that a real-valued band-limited signal can be recovered up to sign from the absolute values of its samples at any rate greater than twice the Nyquist rate. A similar result for complex valued band-limited signals was shown in \cite{PYB14} which required a minimal oversampling rate of four times the Nyquist rate.

In \cite{MW14} the authors study an instance of the phase retrieval problem using the Cauchy wavelet transform to recover analytic functions in $L^2(\mathbb{R},\mathbb{C})$ that have compactly supported Fourier transforms. In that paper they observe that although they are able to show that $\mathcal{A}_{\Phi}$ is injective (for the particular choice of $\mathcal{H}$ and $\Phi$) there is an inherent lack of robustness in the sense that arbitrarily small perturbations of the measurements $\mathcal{A}_{\Phi}(f)$ can result in large errors in the reconstructed signal (see sections 4.1 and 4.2 in \cite{MW14}). The main result of the present paper states that this type of lack of robustness is unavoidable when doing phase retrieval in an infinite-dimensional Hilbert space.

In this paper, we restrict ourselves to the case of countably infinite frames in Hilbert
spaces; in work extending the present results, \cite{AG} proves similar lack of robustness
for phase retrieval in infinite-dimensional Banach spaces with infinite frames that need not be 
countable. 

One way to quantify the robustness of the phase retrieval process for a given frame $\Phi$ is in terms of the lower Lipschitz bound of the map $\mathcal{A}_{\Phi}$ with respect to some metric on the space $\mathcal{H}/\sim$. A natural choice of metric is the quotient metric induced by the metric on $\mathcal{H}$ given by
$$
d(\tilde{f},\tilde{g})=\inf_{|\alpha|=1}\|f-\alpha g\|.
$$

We would like to find a positive constant $C$ (depending only on $\Phi$) so that for every $f,g\in\mathcal{H}$
\begin{equation}\label{Lipschitz}
\inf_{|\alpha|=1}\|f-\alpha g\|\leq C\|\mathcal{A}_{\Phi}(f)-\mathcal{A}_{\Phi}(g)\|.
\end{equation}

In \cite{BCMN14} the authors introduced a numerical version of the complement property as a means of quantifying the constant $C$ in \eqref{Lipschitz}:

\begin{definition}
We say a frame $\{\varphi_n\}_{n\in I}$ has the $\sigma$-strong complement property if for every subset $S\subseteq I$ either $\{\varphi_n\}_{n\in S}$ or $\{\varphi_n\}_{n\not\in S}$ is a frame for $\mathcal{H}$ with lower frame bound at least $\sigma$.
\end{definition}

In \cite{BCMN14} it is shown that when $\mathcal{H}=\mathbb{R}^M$ the lower Lipschitz bound of $\mathcal{A}_{\Phi}$ is precisely controlled by the largest $\sigma$ for which $\Phi$ has the $\sigma$-strong complement property (see also \cite{BY15}). Although this result does not apply to the complex case, much like the complement property cannot be used to determine whether a given frame does phase retrieval for a complex space, we still have the following result in the finite-dimensional case:

\begin{proposition}\label{finite}
If $\mathcal{H}$ is a finite-dimensional Hilbert space and $\Phi=\{\varphi_n\}_{n\in I}$ does phase retrieval for $\mathcal{H}$ then $\mathcal{A}_{\Phi}$ has a positive lower Lipschitz bound, i.e., $\mathcal{A}_{\Phi}$ satisfies \eqref{Lipschitz} for some $C<\infty$.
\end{proposition}
\begin{proof}
Since this result is already known if $\mathcal{H}$ is a real Hilbert space we will prove it for the case where $\mathcal{H}=\mathbb{C}^N$. Note that the inequality \eqref{Lipschitz} is homogeneous so without loss of generality we can assume that $\|f\|=1$ and $\|g\|\leq 1$.

Let $\mathbb{H}_N$ denote the space of $N\times N$ Hermitian matrices equipped with the Hilbert-Schmidt inner product $\langle X,Y\rangle=\mathrm{Trace}(XY)$. Because of the restriction to Hermitian matrices, this is a Hilbert space over the reals (of dimension $M^2$), and no adjoint is necessary in the definition of $\langle X,Y\rangle$. Define the linear mapping $\mathcal{A}_{\Phi}^2:\mathbb{H}_N\rightarrow\ell^2(I)$ by
$$
\mathcal{A}_{\Phi}^2(X)=(|\langle X,\varphi_n\varphi_n^*\rangle|)_{n\in I}=(|\langle X \varphi_n,\varphi_n\rangle|)_{n\in I},
$$
where we denote by $gg^*$ the rank one operator that maps 
$h \in \mathbb{C}^N$ to $\langle h, g\rangle g$.
(Note that if $X$ is rank 1, i.e. $X=ff^*$, then $\mathcal{A}_{\Phi}^2(X)=(|\langle f,\varphi_n\rangle|^2)_{i\in I}$, hence the notation $\mathcal{A}_{\Phi}^2$.) It is well known that $\Phi$ does phase retrieval if and only if $\ker(\mathcal{A}_{\Phi}^2)$ does not contain any matrix of rank 1 or 2 (see Lemma 9 in \cite{BCMN14}). This, together with the compactness of the set $S=\{X\in\mathbb{H}_N:\mathrm{rank}(X)\leq 2,\|X\|=1\}$ (since $\mathbb{H}_N$ is 
finite-dimensional), implies that
$$
\min_{X\in S}\|\mathcal{A}_{\Phi}^2(X)\|=c>0
$$
where $\|X\|$ denotes the operator norm (however, we can choose any norm on $\mathbb{H}_N$ and this will still be true).

For $f,g\in\mathbb{C}^N$, $ff^*-gg^*$ is rank 1 or 2, so we have
\begin{eqnarray}
\notag \|ff^*-gg^*\|^2&\leq&\frac{1}{c^2}\|\mathcal{A}_{\Phi}^2(ff^*-gg^*)\|^2 \\
\notag &=& \frac{1}{c^2}\|\mathcal{A}_{\Phi}^2(ff^*)-\mathcal{A}_{\Phi}^2(gg^*)\|^2 \\
&=&\label{1} \frac{1}{c^2}\sum_{n\in I}(|\langle f,\varphi_n\rangle|^2-|\langle g,\varphi_n\rangle|^2)^2.
\end{eqnarray}
Furthermore, since we are assuming $\|f\|=1$ and $\|g\|\leq 1$ we have
\begin{eqnarray}
\notag \sum_{n\in I}(|\langle f,\varphi_n\rangle|^2-|\langle g,\varphi_n\rangle|^2)^2&=&\sum_{n\in I}(|\langle f,\varphi_n\rangle|-|\langle g,\varphi_n\rangle|)^2(|\langle f,\varphi_n\rangle|+|\langle g,\varphi_n\rangle|)^2 \\
\notag &\leq& \sum_{n\in I}(|\langle f,\varphi_n\rangle|-|\langle g,\varphi_n\rangle|)^2(2\|\varphi_n\|)^2 \\
&\leq&\label{2} (4\max_{n\in I}\|\varphi_n\|^2)\sum_{n\in I}(|\langle f,\varphi_n\rangle|-|\langle g,\varphi_n\rangle|)^2.
\end{eqnarray}

Since we are assuming $\|f\|\geq \|g\|$ a direct computation shows that the largest (in absolute value) eigenvalue of $ff^*-gg^*$ is
\begin{eqnarray*}
\frac{1}{2}(\|f\|^2-\|g\|^2+((\|f\|^2+\|g\|^2)^2-4|\langle f,g\rangle|^2)^{1/2}).
\end{eqnarray*}
Therefore, we have that
\begin{eqnarray*}
\|ff^*-gg^*\|&\geq& \frac{1}{2}((\|f\|^2+\|g\|^2)^2-4|\langle f,g\rangle|^2)^{1/2} \\
&=&\frac{1}{2}(\|f\|^2+\|g\|^2-2|\langle f,g\rangle|)^{1/2}(\|f\|^2+\|g\|^2+2|\langle f,g\rangle|)^{1/2} \\
&=& \frac{1}{2}(\inf_{|\alpha|=1}\|f-\alpha g\|)(\|f\|^2+\|g\|^2+2|\langle f,g\rangle|)^{1/2},
\end{eqnarray*}
and since $\|f\|=1$ this says
\begin{equation}\label{3}
\inf_{|\alpha|=1}\|f-\alpha g\|\leq 2\|ff^*-gg^*\|.
\end{equation}
Finally, combining \eqref{1},\eqref{2}, and \eqref{3} yields \eqref{Lipschitz}.
\end{proof}

\section{Main results}

Before stating the main result we first remark that, when viewed as a subset of $\mathbb{C}^{M\times N}$, the set of frames $\{\varphi_n\}_{n=1}^N$ that do phase retrieval for $\mathbb{C}^M$ is an open subset for each $N$, see \cite{B13,CEHV15}. In fact, in \cite{CEHV15} it is shown that if $N\geq 4M-4$ then this set is open and dense, and it is clear that it must be empty if $N$ is sufficiently small. At this time it is not known if there exists a pair $(M,N)$ where the set of frames consisting of $N$ vectors which do phase retrieval for $\mathbb{C}^M$ is nonempty but not dense (see \cite{V15}), but in any case, the set of frames which do not do phase retrieval is never dense unless it is all of $\mathbb{C}^{M\times N}$. The next statement says that this situation is reversed when we consider frames for an infinite-dimensional space.

\begin{proposition}
Let $\mathcal{H}$ be an infinite-dimensional separable Hilbert space and suppose $\{\varphi_n\}_{n\in\mathbb{N}}$ does phase retrieval.  For every $\epsilon>0$ there is another frame $\{\psi_n\}_{n\in\mathbb{N}}$ which does not do phase retrieval and satisfies
$$
\sum_{n\in\mathbb{N}}\|\varphi_n-\psi_n\|^2<\epsilon.
$$
\end{proposition}
\begin{proof}
Let $\{e_n\}_{n\in\mathbb{N}}$ be an orthonormal basis for $\mathcal{H}$ and choose $k\in\mathbb{N}$ so that
$$
\sum_{n=k+1}^{\infty}|\langle e_1,\varphi_n\rangle|^2<\epsilon.
$$
For $n\leq k$ let $\psi_n=\varphi_n$ and for $n>k$ let
$$
\psi_n=\sum_{i=2}^{\infty}\langle\varphi_n,e_i\rangle e_i.
$$

Now we have that
$$
\sum_{n\in\mathbb{N}}\|\varphi_n-\psi_n\|^2=\sum_{n=k+1}^{\infty}|\langle e_1,\varphi_n\rangle|^2<\epsilon.
$$
Also, it is clear that $\{\psi_n\}_{n=1}^k$ cannot span $\mathcal{H}$, and for every $n>k$ we have that $\langle e_1,\psi_n\rangle=0$, so $\{\psi_n\}_{n=k+1}^{\infty}$ does not span $\mathcal{H}$ either.  Therefore $\{\psi_n\}_{n\in\mathbb{N}}$ does not have the complement property and so by Theorem \ref{CP} does not do phase retrieval.  Furthermore, for $\epsilon$ sufficiently small $\{\psi_n\}_{n\in\mathbb{N}}$ is still a frame.
\end{proof}

The above proposition suggests an infinite-dimensional space is fundamentally different from a finite-dimensional setting
when doing phase retrieval.  In particular, since any frame can be perturbed by an arbitrarily small amount to arrive at a frame that does not do phase retrieval, it suggests that phase retrieval for 
infinite-dimensional spaces is inherently unstable.  We now state the main result,
confirming this intuition.

\begin{theorem}\label{main}
Let $\mathcal{H}$ be an infinite-dimensional separable Hilbert space and let $\Phi=\{\varphi_n\}_{n\in\mathbb{N}}$ be a frame for $\mathcal{H}$ with frame bounds $0<A\leq B<\infty$; further suppose that $\|\varphi_n\|\geq c>0$ for every $n\in\mathbb{N}$. Then, for every $\delta>0$, there exist $f,g\in\mathcal{H}$ so that $\inf_{|\alpha|=1}\|f-\alpha g\|\geq 1$ but $\|\mathcal{A}_{\Phi}(f)-\mathcal{A}_{\Phi}(g)\|<\delta$.
\end{theorem}

Before proving the theorem we need a lemma.

\begin{lemma}
Let $\mathcal{H}$ be an infinite-dimensional separable Hilbert space and let $\Phi=\{\varphi_n\}_{n\in\mathbb{N}}$ be a frame for $\mathcal{H}$ with frame bounds $0<A\leq B<\infty$.  For every $\epsilon>0$ and every $N\in\mathbb{N}$ there is a $k>N$ and a $m>k$ so that
$$
\sum_{n\not\in\{N+1,...,m\}}|\langle \varphi_k,\varphi_n\rangle|^2<\epsilon.
$$
\end{lemma}
\begin{proof}
Fix $\epsilon>0$ and $N\in\mathbb{N}$. Let $V=\mathrm{span}\{\varphi_1,...,\varphi_N\}$ and let $P_V$ denote the orthogonal projection onto $V$.  Let $\{e_{\ell}\}_{\ell=1}^L$ be an orthonormal basis for $V$ and note that
\begin{eqnarray*}
\sum_{n\in N}\|P_V\varphi_n\|^2&=&\sum_{\ell=1}^L\sum_{n\in\mathbb{N}}|\langle\varphi_n,e_{\ell}\rangle|^2 \\
&\leq&\sum_{\ell=1}^LB\|e_{\ell}\|^2=BL.
\end{eqnarray*}
So since $\sum_{n\in\mathbb{N}}\|P_V\varphi_n\|^2<\infty$ we can find $k > N$ so that $\|P_V\varphi_k\|^2<\frac{\epsilon}{2B}$. Then
\begin{eqnarray*}
\sum_{n=1}^N|\langle \varphi_k,\varphi_n\rangle|^2&=&\sum_{n=1}^N|\langle\varphi_k,P_V\varphi_n\rangle|^2 \\
&=&\sum_{n=1}^N|\langle P_V\varphi_k,\varphi_n\rangle|^2 \\
&\leq&\sum_{n\in\mathbb{N}}|\langle P_V\varphi_k,\varphi_n\rangle|^2 \\
&\leq&B\|P_V\varphi_k\|^2<B\frac{\epsilon}{2B}=\frac{\epsilon}{2}.
\end{eqnarray*}

Now observe that
\begin{eqnarray*}
\sum_{n=k+1}^{\infty}|\langle \varphi_k,\varphi_n\rangle|^2&\leq&\sum_{n\in\mathbb{N}}|\langle \varphi_k,\varphi_n\rangle|^2 \\
&\leq& B\|\varphi_k\|^2<\infty,
\end{eqnarray*}
so there is an $m>k$ so that
$$
\sum_{n=m+1}^{\infty}|\langle \varphi_k,\varphi_n\rangle|^2<\frac{\epsilon}{2}.
$$
Therefore
\begin{eqnarray*}
\sum_{n\not\in\{N+1,...,m\}}|\langle \varphi_k,\varphi_n\rangle|^2&=&\sum_{n=1}^N|\langle \varphi_k,\varphi_n\rangle|^2+\sum_{n=m+1}^{\infty}|\langle \varphi_k,\varphi_n\rangle|^2 \\
&<&\frac{\epsilon}{2}+\frac{\epsilon}{2}=\epsilon.
\end{eqnarray*}
\end{proof}

Note that one consequence of the above lemma is that no frame for an infinite-dimensional Hilbert space can have the $\sigma$-strong complement property, regardless of how small one picks $\sigma>0$.

\begin{proof}[Proof of {\rm Theorem \ref{main}}]
We use the Lemma to construct $f$ and $g$ explicitly. 

Pick $\epsilon = c^2\,\delta^2/4 $ and $N \in \N$, and determine $k$ and $m$ as in the Lemma, for these choices of 
$\epsilon,N$. Next, pick $\psi$ so that it is orthogonal to the finite-dimensional span of 
$\varphi_{N+1},\, \ldots,\varphi_m$, and set $f\,=\,\varphi_k\,\|\varphi_k\|^{-1} \,+\, \psi\,$,
and $g\,=\,\varphi_k\,\|\varphi_k\|^{-1} \,-\, \psi\,$. 

For $n \in \{N+1,\ldots,m\}$ we have $\langle f, \varphi_n \rangle = \langle g, \varphi_n \rangle$, 
so that 
\[
\sum_{n \in \Z} \left(|\langle f, \varphi_n \rangle | - |\langle g, \varphi_n \rangle |\right)^2
= \sum_{n \not\in \{N+1,\ldots,m \}} \left(|\langle f, \varphi_n \rangle | - |\langle g, \varphi_n \rangle |\right)^2 \,.
\]
The triangle inequality implies that 
$\left|\,|z_1+z_2|-|z_1-z_2|\,\right| \leq 2 |z_1|$ for all $z_1, z_2 \in \C$; applying this
to each term in the right hand side of the inequality, setting 
$z_1= \langle \varphi_k, \varphi_n \rangle \|\varphi_k\|^{-1}$ and $z_2 = \langle \psi, \varphi_n \rangle$,
leads to
\[
\sum_{n \in \Z} \left(|\langle f, \varphi_n \rangle | - |\langle g, \varphi_n \rangle |\right)^2
\leq 4 \sum_{n \not\in \{N+1,\ldots,m \}} \left|\langle \varphi_k, \varphi_n \rangle \right|^2 \, \|\varphi_k\|^{-2}
\leq \frac{4 \epsilon}{c^2} = \delta^2\,,
\]
or $\|\mathcal{A}_{\Phi}(f)\,-\,\mathcal{A}_{\Phi}(g)\|\ \leq \delta $.

On the other hand, because $\psi$ and $\|\varphi_k\|^{-1}\varphi_k$ are orthogonal unit vectors, we have
that, for all $\alpha \in C$ with $|\alpha|=1$,
\[
\|f-\alpha g\|^2 = \|\,(1 -\alpha) \|\varphi_k\|^{-1} \varphi_k \,+\,(1+\alpha) \psi \,\|^2 = |1-\alpha|^2+|1+\alpha|^2= 4\,,
\]
so that $\inf_{|\alpha|=1} \|f-\alpha g \|=2$. 
\end{proof}

\begin{remark}
Although it is not important here, it may be interesting to note that, 
regardless of how small $\delta$ is, the functions
$f$, $g$ constructed in the proof lie within the closed bounded ball 
with radius $\sqrt{2}$ (in fact $\|f\|=\|g\|=\sqrt{2}$). 
\end{remark}

Since Theorem \ref{main} says that we can never do phase retrieval in a robust way for an infinite-dimensional space, but Theorem \ref{finite} says we can basically always do it for a finite-dimensional space, it seems natural to try to use finite-dimensional approximations when working in an infinite-dimensional setting. Our next theorem 
is a first step in that direction; again, we first establish a lemma.

\begin{lemma}\label{UL}
Let $\mathcal{H}$ be a separable Hilbert space and let $\Phi$ be a frame for $\mathcal{H}$ with frame bounds $0<A\leq B<\infty$, then for every $f,g\in\mathcal{H}$ we have
$$
\|\mathcal{A}_{\Phi}(f)-\mathcal{A}_{\Phi}(g)\|\leq B^{1/2}\inf_{|\alpha|=1}\|f-\alpha g\|.
$$
\end{lemma}
\begin{proof}
First note that
$$
||\langle f,\varphi_n\rangle|-|\langle g,\varphi_n\rangle||\leq|\langle f,\varphi_n\rangle-\langle g,\varphi_n\rangle|
$$
by the reverse triangle inequality.  This means that
$$
\|\mathcal{A}_{\Phi}(f)-\mathcal{A}_{\Phi}(g)\|\leq\|T_{\Phi}f-T_{\Phi}g\|\leq B^{1/2}\|f-g\|.
$$
Since $\mathcal{A}_{\Phi}(\alpha g)=\mathcal{A}_{\phi}(g)$ for any unimodular scalar $\alpha$, we have
$$
\|\mathcal{A}_{\Phi}(f)-\mathcal{A}_{\Phi}(g)\|=\inf_{|\alpha|=1}\|\mathcal{A}_{\Phi}(f)-\mathcal{A}_{\Phi}(\alpha g)\|\leq B^{1/2}\inf_{|\alpha|=1}\|f-\alpha g\|.
$$
\end{proof}

\begin{remark}
Since Theorem \ref{main} says that $\mathcal{A}_{\Phi}$ can never have a positive
lower Lipschitz bound when $\Phi$ is a frame it may seem tempting to ask
whether we can achieve a positive lower bound for a set that does not form
a frame, i.e., a sequence $\Phi$ that does not have an upper frame bound.
While this might be possible, Lemma \ref{UL} tells us that in this case
$\mathcal{A}_{\Phi}$ cannot have a finite upper Lipschitz bound. To see this take $g=0$ in the proof of the Lemma so that $\|\mathcal{A}_{\Phi}(f)-\mathcal{A}_{\Phi}(g)\|=\|\mathcal{A}_{\Phi}(f)\|=\|T_{\Phi}(f)\|$ for every $f\in\mathcal{H}$, then use the fact that $\Phi$ does not have a finite upper frame bound to produce a sequence of unit vectors $\{f_n\}_{n\in\mathbb{N}}$ with $\|T_{\Phi}(f_n)\|\rightarrow\infty$.
\end{remark}

\begin{theorem}\label{approx}
Let $\mathcal{H}$ be an infinite-dimensional separable Hilbert space and let $\Phi=\{\varphi_n\}_{n\in\mathbb{N}}$ be a frame for $\mathcal{H}$ with frame bounds $0<A\leq B<\infty$. For each $m\in\mathbb{N}$ let $V_m$ be a finite-dimensional subspace of $\mathcal{H}$ so that $\dim(V_{m+1})>\dim(V_m)$. Suppose there is an increasing function $G(m)$, with $\lim_{m\rightarrow \infty}G(m)=\infty$, so that the following holds for every $m$: for every $f,g\in V_m$
$$
\inf_{|\alpha|=1}\|f-\alpha g\|\leq G(m)\|\mathcal{A}_{\Phi}(f)-\mathcal{A}_{\Phi}(g)\|.
$$
For $\gamma>1$ and $R>0$ define
$$
\mathcal{B}_{\gamma}(R)=\{f\in\mathcal{H}:\|f-P_mf\|\leq G(m+1)^{-\gamma}R\|f\|\text{ for every }m\in\mathbb{N}\}
$$
where $P_m$ denotes the orthogonal projection onto $V_m$. Then for every $f,g\in\mathcal{B}_{\gamma}(R)$ we have
\begin{equation}\label{approx-ineq}
\inf_{|\alpha|=1}\|f-\alpha g\|\leq C \, \left(\|f\|+\|g\|\right)^{1/\gamma}\,\|\mathcal{A}_{\Phi}(f)-\mathcal{A}_{\Phi}(g)\|^{\frac{\gamma-1}{\gamma}}
\end{equation}
where $C$ depends on only $B,R,\gamma$ and $G(1)$.
\end{theorem}
\begin{proof}
Let $f,g\in\mathcal{B}_{\gamma}(R)$. We start by proving an equality of the type 
\begin{equation}
\label{intermediate}
\inf_{|\alpha|=1}\|f-\alpha g\|\leq C' \left(1+\|f\|+\|g\|\right)\|\mathcal{A}_{\Phi}(f)-\mathcal{A}_{\Phi}(g)\|^{\frac{\gamma-1}{\gamma}}\,;
\end{equation}
equation (\ref{approx-ineq}) will then follow by an amplification trick. 

If $\|\mathcal{A}_{\Phi}(f)-\mathcal{A}_{\Phi}(g)\|\geq R\,B^{1/2}\,G(1)^{-\gamma}$, 
then it follows from
$$
\|f-\alpha g\|\leq \|f\| +\|g\| 
$$ 
that (\ref{intermediate}) is satisfied for 
$C'= C_1:=\left(\,R^{-1} B^{-1/2} G(1)^{\gamma}\, \right)^{\frac{\gamma-1}{\gamma}}$. 

Now 
assume that $\|\mathcal{A}_{\Phi}(f)-\mathcal{A}_{\Phi}(g)\|<R\,B^{1/2}\,G(1)^{-\gamma}$.  
Find then $m$ so that
\begin{equation}
\label{bb}
R\,B^{1/2}\,G(m+1)^{-\gamma}\leq\|\mathcal{A}_{\Phi}(f)
-\mathcal{A}_{\Phi}(g)\|\leq R\,B^{1/2}\,G(m)^{-\gamma}
\end{equation}
(we can always do this since $G$ is increasing and 
$\lim_{m\rightarrow \infty} G(m)= \infty$).  We have that
$$
\|f-\alpha g\|\leq \|f-P_mf\|+\|g-P_mg\|+\|P_mf-\alpha P_mg\|,
$$
and since $P_mf,P_mg\in V_m$ we also have that
\begin{eqnarray*}
\inf_{|\alpha|=1}\|P_mf-\alpha P_mg\|
&\leq& G(m)\|\mathcal{A}_{\Phi}(P_mf)-\mathcal{A}_{\Phi}(P_mg)\| \\
&\leq& G(m)\left(\|\mathcal{A}_{\Phi}(f)-\mathcal{A}_{\Phi}(P_mf)\|
+\|\mathcal{A}_{\Phi}(f)-\mathcal{A}_{\Phi}(g)\|
+\|\mathcal{A}_{\Phi}(g)-\mathcal{A}_{\Phi}(P_mg)\|\right)\\
&\leq& G(m)\left(B^{1/2}\|f-P_mf\|+\|\mathcal{A}_{\Phi}(f)
-\mathcal{A}_{\Phi}(g)\|+B^{1/2}\|g-P_mg\|\right)
\end{eqnarray*}
so that
\begin{equation}
\label{aa}
\inf_{|\alpha|=1}\|f-\alpha g\|
\leq\left(1+G(m)B^{1/2}\right)\,\left(\,\|f-P_mf\|+\|g-P_mg\|\,\right)
\,+\,
G(m)\,\|\mathcal{A}_{\Phi}(f)-\mathcal{A}_{\Phi}(g)\|\,.
\end{equation}
Because $f$ and $g$ are both in 
$\mathcal{B}_{\gamma}(R)$, we have
\[
\|f-P_mf\|+\|g-P_mg\|\leq G(m+1)^{-\gamma} R (\,\|f\|+\|g\|\,) 
\leq \,B^{-1/2}\,\|\mathcal{A}_{\Phi}(f)-\mathcal{A}_{\Phi}(g)\|\,(\,\|f\|+\|g\|\,)\,;
\]
on the other hand, using 
$\|\mathcal{A}_{\Phi}(f)-\mathcal{A}_{\Phi}(g)\|<R\,B^{1/2}\,G(1)^{-\gamma}$ and (\ref{bb}), we derive
\begin{eqnarray*}
1+G(m)B^{1/2} &\leq& R^{\frac{1}{\gamma}}\, B^{\frac{1}{2\gamma}} \,G(1)^{-1}\,
 \|\mathcal{A}_{\Phi}(f)-\mathcal{A}_{\Phi}(g)\|^{-\frac{1}{\gamma}}\,
+ \,B^{1/2}\,R^{\frac{1}{\gamma}}\, B^{\frac{1}{2\gamma}} \,
 \|\mathcal{A}_{\Phi}(f)-\mathcal{A}_{\Phi}(g)\|^{-\frac{1}{\gamma}}\,\\
&\leq& R^{\frac{1}{\gamma}}\,B^{\frac{1}{2\gamma}} \,\left(G(1)^{-1}+B^{1/2}\right)
\|\mathcal{A}_{\Phi}(f)-\mathcal{A}_{\Phi}(g)\|^{-\frac{1}{\gamma}}\,.
\end{eqnarray*}
Finally, we also have, using (\ref{bb}) again,
\begin{eqnarray*}
G(m)\,\|\mathcal{A}_{\Phi}(f)-\mathcal{A}_{\Phi}(g)\|
&\leq& \left( \|\mathcal{A}_{\Phi}(f)
-\mathcal{A}_{\Phi}(g)\|^{-1} \,R\,B^{1/2}\, \right)^{\frac{1}{\gamma}} 
\,\|\mathcal{A}_{\Phi}(f)-\mathcal{A}_{\Phi}(g)\|\\
&=&  R^{\frac{1}{\gamma}} \,B^{\frac{1}{2\gamma}} 
\,\|\mathcal{A}_{\Phi}(f)-\mathcal{A}_{\Phi}(g)\|^{\frac{\gamma-1}{\gamma}}
\end{eqnarray*}
Substituting all this into (\ref{aa}), we obtain
\[
\inf_{|\alpha|=1}\|f-\alpha g\|
\leq 
R^{\frac{1}{\gamma}}\,B^{\frac{1}{2\gamma}} 
\left[\,\left(B^{-1/2}G(1)^{-1}+1\right)\,(\,\|f\|+\|g\|\,)\,
+ \,1\,\right]\, 
\|\mathcal{A}_{\Phi}(f)-\mathcal{A}_{\Phi}(g)\|^{\frac{\gamma-1}{\gamma}}~,
\]
which does indeed imply the inequality (\ref{intermediate}), with 
$C'=C_2:= R^{\frac{1}{\gamma}}\,B^{\frac{1}{2\gamma}} \left(B^{-1/2}G(1)^{-1}+1\right)$.

It thus follows that, for all $f,g\in\mathcal{B}_{\gamma}(R)$, (\ref{intermediate}) holds
for $C'=\max(C_1,C_2)$; $C'$ is completely determined by $R,B,\gamma$ and $G(1)$. 

This inequality can be sharpened further by exploiting its non-homogeneous nature. The set
$\mathcal{B}_{\gamma}(R)$ is invariant under scaling: if $f \in \mathcal{B}_{\gamma}(R)$, 
then so are all multiples of $f$. If, given $f,g\in\mathcal{B}_{\gamma}(R)$, 
we write the inequality (\ref{intermediate}) for $f'=Mf$, $g'=Mg$, where $M \in \R_+$ is
to be fixed below, then we find, upon dividing both sides by $M \neq 0$, 
\begin{eqnarray*}
\inf_{|\alpha|=1}\|f-\alpha g\|
&\leq& C'\, M^{-1}\, \left[\,1\,+\,M\,(\|f\|+\|g\|)\,\right]
M^{\frac{\gamma-1}{\gamma}}\,
\|\mathcal{A}_{\Phi}(f)-\mathcal{A}_{\Phi}(g)\|^{\frac{\gamma-1}{\gamma}}\\
&=& C'\,M^{-\frac{1}{\gamma}}\,\left[\,1\,+\,M\,(\|f\|+\|g\|)\,\right]
\|\mathcal{A}_{\Phi}(f)-\mathcal{A}_{\Phi}(g)\|^{\frac{\gamma-1}{\gamma}}~.
\end{eqnarray*}
Since this inequality holds for all $M \in \R_+$, it holds in particular
for the value $M=\left[(\gamma -1)\left(\|f\|+\|g\|\right) \right]^{-1}$ that minimizes the 
right hand side. We obtain
\[
\inf_{|\alpha|=1}\|f-\alpha g\|
\leq C'\, \gamma \, (\gamma -1)^{\frac{1-\gamma}{\gamma}}\, 
\left(\|f\|+\|g\|\right)^{\frac{1}{\gamma}}\, 
\|\mathcal{A}_{\Phi}(f)-\mathcal{A}_{\Phi}(g)\|^{\frac{\gamma-1}{\gamma}}
\]
\end{proof}
\begin{remark}
We note that although {\rm (\ref{approx-ineq})} describes a type of H\"{o}lder continuity 
(and thus uniform continuity) for
$\mathcal{A}_{\Phi}$, when restricted to 
$\mathcal{B}_{\gamma}(R)$, with H\"{o}lder exponent $(\gamma-1)/\gamma$,
it does not establish Lipschitz continuity (which would require H\"{o}lder exponent 1). 
So far, most papers on the stability of phase
retrieval have focused on showing Lipschitz continuity;
we do not know whether Lipschitz bounds are possible within our
framework, or whether these
weaker H\"{o}lder-type bounds are the strongest possible here.
\end{remark}

This theorem complements Theorem \ref{main}: even though
uniformly stable phase retrieval is never possible in 
infinite-dimensional $\mathcal{H}$, Theorem \ref{approx}
establishes that stable phase retrieval is possible for elements of
$\mathcal{H}$ that can be approximated sufficiently well by
finite-dimensional expansions, and quantifies the ``extent''
of this restricted stability. 

Note that we did not require $\Phi$ to do phase retrieval for $\mathcal{H}$ in the statement of Theorem \ref{approx}. As the following proposition shows, this is, in fact, not necessary.

\begin{proposition}\label{Riesz}
For any infinite-dimensional separable Hilbert space $\mathcal{H}$ there exists a Riesz basis $\Phi=\{\varphi_n\}_{n\in\mathbb{N}}$ and a sequence of subspaces $\{V_m\}$ so that $(\Phi,\{V_m\})$ satisfy the hypotheses of \rm Theorem {\rm\ref{approx}}.
\end{proposition}
\begin{proof}
Let $\{e_n\}_{n\in\mathbb{N}}$ be an orthonormal basis for $\mathcal{H}$ and for each $m\in\mathbb{N}$ let $V_m=\mathrm{span}\{e_n\}_{n=1}^m$.  For each $m$ choose  $\Psi_m$ to be a finite set of vectors in $V_m$ that does phase retrieval for $V_m$; we know that 
$d_m:=\# V_m \geq r\,m$, where $r=2$ or 
$r=4$, according to whether 
$\mathcal{H}$ is a real or complex Hilbert space. Now number the vectors in the successive $\Psi_m$ consecutively, starting at $n=2$, 
so that $\{\psi_n\}_{n=D_{m-1}+2}^{D_{m}+1} = \Psi_m$ for each $m \in \mathbb{N}$, where
$D_m=\sum_{k=1}^{m} d_k$ for $m \geq 1$ and $D_0=0$. Without loss of generality, we can normalize the
vectors in each $\Psi_m$ such that
$$
\sum_{n=2}^\infty\|\psi_n\|^2<\epsilon\,.
$$
with $\epsilon$ to be fixed below.
If we let $\varphi_n=e_n+\psi_n$ for $n \geq 2$ and $\varphi_1=e_1$, 
then
$$
\sum_{n\in\mathbb{N}}\|\varphi_n-e_n\|^2=\sum_{n=2}^\infty\|\psi_n\|^2<\epsilon,
$$
so for $\epsilon$ sufficiently small $\Phi=\{\varphi_n\}_{n\in\mathbb{N}}$ is a Riesz basis for $\mathcal{H}$ (see, e.g., Theorem 15.3.2 in \cite{Ole}).

Also, if $f\in V_m$ then  $\langle f,e_n\rangle=0$ for  
$n \in \{D_{m-1}+2,\ldots ,D_m,D_m+1  \}$,
since $D_{m-1}+1 \geq 1+r/2 \,m(m-1) \geq m$; consequently $\langle f,\varphi_n\rangle=\langle f,\psi_n\rangle$ for $f\in V_m$ and $D_{m-1}+1<n\leq D_{m}+1$. This means that $\{\varphi_n\}_{n=D_{m-1}+2}^{D_{m}+1}$ does phase retrieval for $V_m$, so we can apply Proposition \ref{finite} to find $C_m<\infty$ so that
$$
\inf_{|\alpha|=1}\|f-\alpha g\|\leq C_m\|\mathcal{A}_{\Phi}(f)-\mathcal{A}_{\Phi}(g)\|
$$
for every $f,g\in V_m$. Since this can be done for each $m\in\mathbb{N}$ it follows that $(\Phi,\{V_m\})$ satisfies Theorem \ref{approx} with $G(m)=\max_{1\leq k\leq m}C_k$.
\end{proof}

Since Riesz bases {\em never} do phase retrieval, Proposition 
\ref{Riesz} does indeed show that a frame in an infinite-dimensional 
Hilbert space need not do phase retrieval itself in order to satisfy all the conditions in Theorem \ref{approx}. On the other hand, if $\Phi$ does phase retrieval for $\mathcal{H}$ and $\{V_m\}_{m\in\mathbb{N}}$ is any sequence of finite-dimensional subspaces with increasing dimensions, then it follows from Proposition \ref{finite} that 
the pair $(\Phi,\{V_m\}_{m\in\mathbb{N}})$ satisfies the hypotheses of Theorem \ref{approx} for some function $G(m)$; see
also Proposition \ref{propp} below.

In the formulation of Theorem \ref{approx}, we used the infinite 
sequences 
$\mathcal{A}_{\Phi}(f):= \left(|\langle f, \varphi_n \rangle| \right)_{n \in \mathbb{N}}$, for $f \in V_m$, even though this is 
surely overkill for elements $f$ in 
the finite-dimensional spaces $V_m$. 
As the following proposition shows, one can, at little cost, restrict 
the frame to an appropriate finitely-truncated set of vectors for each $V_m$: if $\{\Phi,\{V_m\}_{m\in\mathbb{N}}\}$ satisfy Theorem \ref{approx} for the function $G(m)$, then there is a $N(m)$ for each $m$ so that the $(\{\varphi_n\}_{n=1}^{N(m)},\{V_m\}_{m\in\mathbb{N}})$ satisfy Theorem \ref{approx} for some function $H(m)\geq G(m)$. 

\begin{proposition}\label{propp}
Let $\mathcal{H}$ be an infinite-dimensional separable Hilbert space and let $\{\varphi_n\}_{n\in\mathbb{N}}$ be a frame for $\mathcal{H}$.  Let $V$ be a finite-dimensional subspace of $\mathcal{H}$ and suppose $\{P_V\varphi_n\}_{n\in\mathbb{N}}$ does phase retrieval for $V$, then there is an $N(V)<\infty$ so that $\{P_V\varphi_n\}_{n=1}^{N(V)}$ does phase retrieval for $V$.
\end{proposition}
\begin{proof}
For notational convenience let $\psi_n=P_V\varphi_n$ for each $n$ and let $\Psi=\{\psi_n\}_{n\in\mathbb{N}}$. Also, suppose $\dim(V)=M$. Recall from the proof of Proposition \ref{finite} that $\Psi$ does phase retrieval (for V) if and only if $\ker{\mathcal{A}_{\Psi}^2}$ does not contain any rank 1 or 2 matrices. But $\ker{\mathcal{A}_{\Psi}^2}=\mathrm{span}\{\psi_n\psi_n^*\}_{n\in\mathbb{N}}^{\perp}$, and $N:=\dim(\mathrm{span}\{\psi_n\psi_n^*\}_{n\in\mathbb{N}})\leq\dim(\mathbb{H}_M)=M^2$. Therefore there exists a subset $I\subset\mathbb{N}$ with $|I|=N$ and $\mathrm{span}\{\psi_n\psi_n^*\}_{n\in I}=\mathrm{span}\{\psi_n\psi_n^*\}_{n\in\mathbb{N}}$, which means $\ker\mathcal{A}_{\Psi_I}^2=\ker\mathcal{A}_{\Psi}^2$ and so $\{\psi_n\}_{n\in I}$ does phase retrieval for $V$.
\end{proof}

Note that in the requirement that 
$\{P_V\varphi_n\}_{n\in\mathbb{N}}$ does phase retrieval for $V$, 
i.e. that, up to global phase, any $f \in V$ can be reconstructed from
the sequence $ \left(|\langle f, P_V \varphi_n \rangle | \right)_{n \in \mathbb{N}}$, one can equally well drop the projector $P_V$, since
$\langle f, P_V \varphi_n \rangle =\langle f, \varphi_n \rangle  $. 

The frame used in the proof of Proposition \ref{Riesz} was rather
contrived, and would not be used in any concrete applications. It is 
reasonable to wonder how the function $G$ of Theorem \ref{approx}
behaves for choices of
$\Phi$ and $\{V_m\}$ of practical interest. It turns out  
that it may grow very fast, as illustrated by the following example.
\begin{example}
Consider 
the space $\mathcal{H}$ of real, square integrable functions with bandlimit $\pi$, i.e. with Fourier transform supported on 
$[-\pi,\pi]$, and the functions $\varphi_n \in \mathcal{H}$ defined, for $n \in \mathbb{Z}$, by
\begin{eqnarray*}
\varphi_n(x)=&\frac{\sin\left(\pi(x-\frac{n}{4})\right)}{\pi(x-\frac{n}{4})}&\mbox{ if }~~ x\neq \frac{n}{4}\\
&1  &\mbox{ if }~~ x= \frac{n}{4}~~.
\end{eqnarray*}
Note that the $\varphi_{4\ell}$, $\ell \in \mathbb{Z}$ constitute
the standard ``Shannon'' orthonormal basis of $\mathcal{H}$.
As shown in {\rm \cite{T11}}, $\Phi:=\{\varphi_n\}_{n \in \mathbb{Z}}$ does phase retrieval for
$\mathcal{H}$.\\
Define the spaces $V_n \subset \mathcal{H}$ as the $(2n+1)$-dimensional
subspaces of $\mathcal{H}$ spanned by the $\varphi_{4\ell}$, $-n\leq \ell \leq n$. Then there exist
functions $f_m$, $g_m$ in
$V_{2m}$ with sufficiently fast inverse polynomial decay such
that, for some $m$-independent constant $C>0$, the following
inequality holds for all $m\in \mathbb{N}$:
\begin{equation}\label{bd_below}
\inf_{\alpha \in \{1,-1\}} \|f_m-\alpha g_m\| > C (m+1)^{-1} 2^{3 m} \, 
\|\mathcal{A}_{\Phi}(f)-\mathcal{A}_{\Phi}(g)\|.
\end{equation}
\end{example}

\begin{proof}
We construct $f_m$ and $g_m$ explicitly. 

We start by defining, for
$k \in \mathbb{N}$, the functions $s_k \in \mathcal{H}$ recursively
by setting $s_0(x)=\varphi_0(x)$, and, for $k>0$, 
$s_k(x):=s_{k-1}(x+1)+s_{k-1}(x)$; equivalently
$s_k(x)= \sum_{\ell=0}^k \binom{k}{l} \varphi_0(x+\ell)$. By induction on $k$, one easily
checks that, for $x \not\in \{0,-1,\ldots,-k\}$,
\begin{equation}\label{explicit}
s_k(x)=\frac{1}{\pi}\, \frac{k!\, \sin \pi x}{x(x+1)\ldots(x+k)}\,;
\end{equation}
for $x\in \{0,-1,\ldots,-k\}$, i.e. $x=-n$ with $n\in \{0,1,\ldots,k\}$,
one has $s_k(-n)=\binom{k}{n}$. 

Next, we define
\[
f_m(x):=s_m(x+m)+s_m(x-2m)~~ \mbox{ and } ~~ g_m(x):=s_m(x+m)-s_m(x-2m)~;
\]
equivalently, in a more symmetric form that makes clear that  
$f_m$, $g_m$ are both in $V_{2m}$,
\begin{eqnarray*}
f_m(\cdot)&=&\sum_{\ell=0}^m \binom{m}{\ell}\left[\varphi_0(\cdot+2m-\ell)+ \varphi_0(\cdot-2m+\ell)\right]= \sum_{\ell=0}^m \binom{m}{\ell}
\left[\varphi_{4(-2m+\ell)} + \varphi_{4(2m-\ell)} \right]\\
g_m(\cdot)&=&\sum_{\ell=0}^m \binom{m}{\ell}\left[\varphi_0(\cdot+2m-\ell)- \varphi_0(\cdot-2m+\ell)\right]= \sum_{\ell=0}^m \binom{m}{\ell}
\left[\varphi_{4(-2m+\ell)} - \varphi_{4(2m-\ell)}  \right]~.
\end{eqnarray*}

Then 
\[
\inf_{\alpha \in \{1,-1\}} \|f_m-\alpha g_m\|^2 =
4 \min \left( \|s_m(\cdot-2m)\|^2,\|s_m(\cdot+m)\|^2 \right) = 
4 \sum_{\ell=0}^m \left[ \binom{m}{\ell} \right]^2 ~.
\]
Combining the identity $\sum_{\ell=0}^m \binom{m}{\ell} = 2^m$  with
$\sum_{\ell=0}^m \left(a_\ell\right)^2 \leq \left[\sum_{\ell=0}^m a_\ell \right]^2 \leq (m+1) \sum_{\ell=0}^m \left(a_\ell\right)^2$, valid for 
all non-negative numbers $a_0,\ldots,a_m$, and making use of the 
orthonormality of the $\varphi_{4n}$ in $\mathcal{H}$, we have thus
\begin{equation}\label{ex_1}
\inf_{\alpha \in \{1,-1\}} \|f_m-\alpha g_m\|^2 \geq 4 (m+1)^{-1} 2^{2m} ~.
\end{equation}

We now estimate $\|\mathcal{A}_{\Phi}(f_m)-\mathcal{A}_{\Phi}(g_m)\|^2$. 
Since, for all $h \in \mathcal{H}$, $\langle h, \varphi_\ell \rangle=
h(\ell/4)$, we have
\begin{eqnarray*}
\|\mathcal{A}_{\Phi}(f_m)-\mathcal{A}_{\Phi}(g_m)\|^2
&=& \sum_{n \in \mathbb{Z}}\left(\left|f_m(n/4)\right| - \left| g_m(n/4) \right|  \right)^2 \\
&=& \sum_{k \in \mathbb{Z}}\left[
\left(\left|f_m(k)\right| - \left| g_m(k) \right|\right)^2
+\left(\left|f_m\left(k+\frac{1}{4}\right)\right| - \left| g_m\left(k+\frac{1}{4}\right)\right|\right)^2 \right.\\
&&\left.
+\left(\left|f_m\left(k+\frac{1}{2}\right)\right| - \left| g_m\left(k+\frac{1}{2}\right) \right|\right)^2  
+ \left(\left|f_m\left(k-\frac{1}{4}\right)\right| - \left| g_m\left(k-\frac{1}{4}\right) \right|\right)^2
\right]\\
&=:& T_0 + T_1 +T_2 +T_{-1} ~.
\end{eqnarray*}
Since $s_m(k+m)=0$ unless $k\in\{-m,-m-1,\ldots,-2m\}$ and 
$s_m(k-2m)=0$ unless $k\in\{2m, 2m-1,\ldots,m\}$, we have $|f_m(k)|
=|g_m(k)|$ for all $k \in \mathbb{Z}$, so that $T_0=0$. 

To estimate the other three terms $T_1$, $T_2$ and $T_{-1}$, we first
observe that, for $a,b \in \mathbb{R}$, 
\[
(|a+b|-|a-b|)^2 = 4 \min(|a|,|b|)^2~.
\]
Using the explicit formula (\ref{explicit}) for $s_k(x)$, we can also rewrite $f_m$ and $g_m$ as 
\begin{eqnarray*}
f_m(x)&=&\frac{m!\sin \pi x }{\pi}\left[\frac{(-1)^m}{(x+m)(x+m+1) \ldots (x+2m)} + \frac{1}{(x-m)(x-m-1)\ldots(x-2m)} \right]\\
g_m(x)&=&\frac{m!\sin \pi x }{\pi}\left[\frac{(-1)^m}{(x+m)(x+m+1) \ldots (x+2m)} - \frac{1}{(x-m)(x-m-1)\ldots(x-2m)} \right]~.
\end{eqnarray*}
Combining these two observations, we find that 
\[
\left(\left|f_m(x)\right| - \left| g_m(x) \right|\right)^2
= \frac{4 (m!)^2 (\sin \pi x)^2}{\pi^2 \prod_{s=m}^{2m}(|x|+s)^2}~, 
\]
and thus
\begin{eqnarray*}
\|\mathcal{A}_{\Phi}(f_m)-\mathcal{A}_{\Phi}(g_m)\|^2 &=&
\frac{4 (m!)^2}{\pi^2}\sum_{k\in\mathbb{Z}}
\left [\frac{1}{\prod_{s=m}^{2m}(|k+1/2|+s)^2} +
\frac{1}{\prod_{s=m}^{2m}(|k+1/4|+s)^2}\right]\\
&\leq&\frac{8 (m!)^2}{\pi^2}\sum_{k=0}^{\infty}\left [\frac{1}{\prod_{s=m}^{2m}(k+1/2+s)^2} +
\frac{1}{\prod_{s=m}^{2m}(k+1/4+s)^2}\right]\\
&\leq&\frac{16 (m!)^2}{\pi^2}\sum_{k=0}^{\infty}\frac{1}{\prod_{s=m}^{2m}(k+s)^2}
\leq\frac{16 (m!)^2}{\pi^2}\sum_{k=0}^{\infty}\frac{[m!]^2}{(m+k)^2 [(2m)!]^2} \\
&\leq& \frac{16}{\pi^2} \left[\binom{2m}{m} \right]^{-2} \sum_{k=0}^{\infty}\frac{1}{(1+k)^2} 
\leq \frac{32}{\pi^2} \left[\binom{2m}{m} \right]^{-2}~.
\end{eqnarray*}
By Stirling's formula, $\binom{2m}{m}\sim  S\, m^{-1/2}\,2^{2m}$ as $m \rightarrow \infty$, for some $m$-independent $S>0$, so that, for sufficiently large $m$, 
\[
\|\mathcal{A}_{\Phi}(f_m)-\mathcal{A}_{\Phi}(g_m)\|^2 \leq \frac{36 \,m}{\pi^2 S^2} \,2^{-4m}~,
\]
and consequently, combining this with (\ref{ex_1}),
\[
\frac{\inf_{\alpha \in \{1,-1\}} \|f_m-\alpha g_m\|}
{\|\mathcal{A}_{\Phi}(f_m)-\mathcal{A}_{\Phi}(g_m)\|} \geq 
\frac{\pi S\cdot 2^{3m}}{\,3\,(m+1)}~. 
\]
\end{proof}

Note that since the frame $\Phi$ in this example does phase retrieval for $\mathcal{H}$ (see \cite{T11}), it must do phase retrieval for each $V_m \subset \mathcal{H}$, meaning that (by the argument used in the proof of Proposition \ref{propp}) there must indeed exist $N(m)$
and $H(m)$
such that, for all $f,g \in V_m$
\begin{eqnarray*}
\inf_{|\alpha|=1}\|f-\alpha g\|
&\leq & H(m)\,\left[ \sum_{n=1}^{N(m)}\left(\left| \langle f, \varphi_n \rangle \right| -  \left| \langle g, \varphi_n \rangle  \right|   \right)^2   \right]^{1/2}\\
&\leq & H(m)\,\|\mathcal{A}_{\Phi}(f)-\mathcal{A}_{\Phi}(g)\|~,
\end{eqnarray*}
i.e. (\ref{approx-ineq}) must be satisfied for 
$G(m)=\max_{1\leq k\leq m}H(k)$. 
The computation above tells us that, no matter how large we pick the
$N(m)$, 
$G$ must
grow at least as fast as indicated by (\ref{bd_below}). 

{\bf Acknowledgments}\\
The authors wish to thank the reviewer for a careful reading of the paper, 
pointing out several inaccuracies, and suggesting a shortcut for the
proof of Theorem \ref{main}.

\end{document}